\documentclass[10pt]{amsart}
\usepackage{amsmath}
\usepackage{amssymb}
\usepackage{amsfonts}
\usepackage{amsthm}
\usepackage{enumerate}
\usepackage[all]{xy}
\usepackage[latin1]{inputenc}
\usepackage{mathdots}
\usepackage{graphicx}
\usepackage{latexsym}
\usepackage{mathrsfs}

\usepackage{tikz}

\PassOptionsToPackage{usenames,dvipsnames,svgnames}{xcolor}

\input xy

\newtheorem{Theo}{Theorem}[section]
\newtheorem*{Theor}{Theorem}
\newtheorem{Prop}[Theo]{Proposition}
\newtheorem*{Propo}{Proposition}
\newtheorem{Cor}[Theo]{Corollary}
\newtheorem{Lemma}[Theo]{Lemma}

\newtheorem{Conj}[Theo]{Conjecture}

\theoremstyle{definition}

\newtheorem{Remark}[Theo]{Remark}

\def\mystrut(#1,#2){\vrule height #1pt depth #2pt width 0pt}

\newcommand{\Hom}{{\rm Hom}}
\newcommand{\Ext}{{\rm Ext}}

\newcommand{\mmod}{{\rm mod}\,}

\newcommand{\Z}{\mathbb{Z}}
\newcommand{\C}{\mathcal{C}}

\newcommand{\T}{\mathcal{T}}

\begin{document}

\title[Homological behavior of idempotent subalgebras]{Homological behavior of idempotent subalgebras and Ext algebras}

\author{Charles Paquette}
\address{Department of Mathematics and Computer Science, Royal Military College of Canada, Kingston, ON, K7K 7B4, Canada}
\email{charles.paquette@usherbrooke.ca}

\author{Colin Ingalls}
\address{Department of Mathematics and Statistics, University of New Brunswick, Fredericton, NB, E3B 5A3, Canada}
\email{cingalls@unb.ca}

\thanks{The first author was partially supported by an NSERC Discovery Grant. The second author was supported by the University of Connecticut and partially by the NSF CAREER grant DMS-1254567.}

\begin{abstract} Let $A$ be a (left and right) Noetherian ring that is semiperfect. Let $e$ be an idempotent of $A$ and consider the ring $\Gamma:=(1-e)A(1-e)$ and the semi-simple right $A$-module $S_e : = eA/e\,{\rm rad}A$. In this paper, we investigate the relationship between the global dimensions of $A$ and $\Gamma$, by using the homological properties of $S_e$. More precisely, we consider the Yoneda ring $Y(e):=\Ext^*_A(S_e,S_e)$ of $e$. We prove that if $Y(e)$ is Artinian of finite global dimension, then $A$ has finite global dimension if and only if so does $\Gamma$. We also investigate the situation where both $A,\Gamma$ have finite global dimension. When $A$ is Koszul and finite dimensional, this implies that $Y(e)$ has finite global dimension. We end the paper with a reduction technique to compute the Cartan determiant of Artin algebras. We prove that if $Y(e)$ has finite global dimension, then the Cartan determinants of $A$ and $\Gamma$ coincide. This provides a new way to approach the long-standing Cartan determinant conjecture.
\end{abstract}

\maketitle

\section{Introduction}

Given any ring or algebra $A$, a fundamental question to ask is whether $A$ has finite global dimension or not. In algebraic geometry, the finiteness of the global dimension of an algebra is often associated with the smoothness of a geometric object attached to the algebra. Algebras of finite global dimension appear in the study of non-commutative resolutions. In homological algebra, if $A$ has finite global dimension, then the bounded derived category of the category of $A$-modules can be replaced by the homotopy category of bounded complexes of projective modules, the latter being easier to study. A result of Happel \cite{Happel} states that for a finite dimensional algebra $A$ over a field $k$, the global dimension of $A$ is finite if and only if the bounded derived category of finite dimensional modules has Serre duality. In commutative algebra, when $A$ is Noetherian, Auslander and Buchsbaum have proven the well known fact that $A$ is regular if and only if its global dimension is finite. This result has been proven independently by Serre in \cite{Serre}.

\medskip

In this paper, $A$ is an associative ring that we assume to be (left and right) Noetherian and semiperfect. We investigate the relationship between the global dimensions of the rings $A$ and $\Gamma:=(1-e)A(1-e)$ where $e$ is any idempotent of $A$. In general, the finiteness of the global dimension of $A$ is not equivalent to the finiteness of the global dimension of $\Gamma$. However, it seems that the semi-simple right $A$-module $S_e$ supported at $e$ controls this relationship. We consider the Yoneda ring $Y(e):=\Ext^*_A(S_e, S_e)$ of $e$ and our first theorem is the following.

\begin{Theor}
Assume that $Y(e)$ is (left and right) Artinian and has finite global dimension. Then $A$ has finite global dimension if and only if $\Gamma$ has finite global dimension.
\end{Theor}

If both $A$ and $\Gamma$ have finite global dimensions, then of course, $Y(e)$ is Artinian. However, $Y(e)$ need not have finite global dimension. We suspect, however, that by imposing an acyclicity condition on $e$, this should be true; see Conjecture \ref{conj}. When $A$ is finite dimensional positively graded such that $S_e$ has a linear projective resolution, we get the following.

\begin{Theor}
Assume that $A$ is finite dimensional positively graded such that $S_e$ has a linear projective resolution. Then any two of the following conditions imply the third.
\begin{enumerate}[$(1)$]
    \item The global dimension of $A$ is finite.
\item The global dimension of $\Gamma$ is finite.
\item The algebra $Y(e)$ is finite dimensional and has finite global dimension. \end{enumerate}
\end{Theor}

The above theorem applies to finite dimensional Koszul algebras, since every semi-simple module over a Koszul algebra has a linear projective resolution. We end this paper with a new reduction technique to compute the Cartan determinant of an algebra. In particular, this gives a new tool to attack the long-standing Cartan determinant conjecture.

\begin{Propo}
Assume that $A$ is an artin algebra with both $A$ and $Y(e)$ of finite global dimension. Then $A$ and $\Gamma$ have the same Cartan determinant and $\Gamma$ has finite global dimension. In particular, if the Cartan determinant conjecture holds for the smaller algebra $\Gamma$, then it holds for $A$.
\end{Propo}

\section{Preliminaries and notations}

Let $A$ be an associative (left and right) Noetherian ring that is semiperfect. In particular, $A$ may be an artin algebra. We refer the reader to \cite{AndFul} for more details on semiperfect rings. In particular, there is a complete set of pairwise orthogonal primitive idempotents $e_1, \ldots, e_n$ of $A$. Since we are mainly concerned with homological algebra, there is no loss of generality in assuming that $A$ is basic, that is, the projective right $A$-modules $e_iA$ for $1 \le i \le n$ are pairwise non-isomorphic. We denote by $\mmod A$ the category of finitely generated right $A$-modules and by ${\rm rad} A$ the Jacobson radical of $A$. Since $A$ is Noetherian semiperfect, any $M \in \mmod A$ admits a minimal projective resolution whose terms are in $\mmod A$. Observe also that the category $\mmod A$ is Krull-Schmidt. In particular, an indecomposable projective $A$-module in $\mmod A$ has to be isomorphic to one of the $e_iA$ and an indecomposable simple $A$-module in $\mmod A$ has to be isomorphic to one of the $e_jA/e_j{\rm rad}A$.

\medskip

In this paper, $e$ denotes an idempotent of $A$. The \emph{rank} of $e$, denoted ${\rm rk}(e)$, is an integer between $0$ and $n$ such that ${\rm rk}(1)=n$ and if $e = e_1 + e_2$ with $e_1, e_2$ orthogonal idempotents, then ${\rm rk}(e) = {\rm rk}(e_1) + {\rm rk}(e_2)$. In particular, $e$ is primitive if and only if ${\rm rk}(e)=1$. We set $\Gamma = (1-e)A(1-e) = {\rm End}_A((1-e)A)$, which we call an \emph{idempotent subring} (or \emph{idempotent subalgebra} if a ground commutative ring is given) of $A$. Observe that $\Gamma$ is basic semiperfect Noetherian (see \cite[Cor. 27.7]{AndFul} and \cite[Prop. 2.3]{Sando}) and has $n-{\rm rk}(e)$ non-isomorphic simple right $\Gamma$-modules. In this sense, when $e \ne 0$, the ring $\Gamma$ is smaller and could be easier to understand from the homological point of view.

\medskip

To relate the rings $A$ and $\Gamma$, we need to introduce a third one. We set $$Y(e) := \Ext_A^*(S_e,S_e) = \textstyle{\bigoplus}_{i \ge 0}\Ext^i_A(S_e, S_e),$$
 where $S_e$ is the semi-simple right $A$-module $S_e: = eA/e{\rm rad}\,A$. By convention, $\Hom(-,-)= \Ext^0_A(-,-)$. The ring $Y(e)$ is the\emph{Yoneda ring} or \emph{Ext-ring} of $S_e$. The multiplication in this ring is induced by the Yoneda product of exact sequences. If $e=1$, we simply call $Y(1)$ the \emph{Yoneda ring} or \emph{Ext-ring} of $A$. Observe that $Y(e)$ is naturally graded by the Ext-degree making it a positively graded ring
such that each graded piece $\Ext^i_A(S_e, S_e)$ is of finite length as an $\Hom_A(S_e, S_e)$-module. It is a (left and right) Artinian ring whenever $A$ has finite global dimension. Although we will consider right $A$ or $\Gamma$ modules, it will be useful for us to study left $Y(e)$-modules. Note that $Y(e)$ has ${\rm rk}(e)$ non-isomorphic simple graded left (or right) $Y(e)$-modules.
We will denote by $\mmod Y(e)$ the category of locally finite graded left $Y(e)$-modules. Recall that $M \in \mmod Y(e)$ if
$M = \textstyle{\bigoplus}_{j \in \Z}M(j)$ where for $i \ge 0, j \in \Z$, we have $\Ext^i_A(S_e, S_e)\cdot M(j) \subseteq M(i+j)$ and $M(j)$ is of finite length as an $\Hom_A(S_e,S_e)$-module. We want to stress the fact that $Y(e)$ is a ring, and is not the shift by $e$ of $Y$. The symbol $Y$ alone will never be used in this paper.

\medskip

Most of our homological computations will be made in derived categories. Let us introduce some notations concerning these derived categories.
First, we denote by $D(A)$ the bounded derived category of $\mmod A$ and by $D(\Gamma)$ the bounded derived category of $\mmod \Gamma$. We let $F$ denote the exact functor of triangulated categories $F= R\Hom_A((1-e)A,-): D(A) \to D^b(\mmod \Gamma)$. If $M \in D(A)$ is given by a complex of projective $A$-modules, since $(1-e)A$ is projective, $F(M)$ is obtained by applying $\Hom_A((1-e)A,-)$ to each term and each differential of the complex $M$. Moreover, when doing so, $F(M)$ is given by a complex of projective $\Gamma$-modules in $D(\Gamma)$. Therefore, $F$ induces a functor $F: {\rm per}A \to {\rm per}\Gamma$ between the corresponding perfect derived categories. Let $D(R\Hom(S_e, S_e))$ denote the derived category of differential graded (written as dg for short)-modules over the dg-ring $R\Hom(S_e, S_e)$. We denote by $G: D(A) \to \mmod Y(e)$ the functor defined as the composite of $R\Hom(-, S_e): D(A) \to D(R\Hom(S_e, S_e))$ with $H: D(R\Hom(S_e, S_e)) \to \mmod Y(e)$, the latter being the functor which associates to a dg-module its total cohomology. Alternatively, we have $$G = \textstyle{\bigoplus}_{i \in \Z}\Hom_{D(A)}(-, S_e[i]): D(A) \to \mmod Y(e).$$ Let $M \in D(A)$.
The degree $i$ part $G(M)(i)$ of $G(M)$ is $$G(M)(i) = \Hom_{D(A)}(M,S_e[i])$$
which is of finite length as an $\Hom_A(S_e, S_e)$-module.
Let $f: M \to N$ be a morphism in $D(A)$. The degree $i$ part of the morphism $G(f): G(N) \to G(M)$ is $G(f)(i) = \Hom_{D( A)}(f, S_e[i])$. It is straightforward to check that $G(f) = (G(f)(i))_{i \in \Z}$ defines a morphism of graded $Y(e)$-modules.
Let add$(S_e)$ denote the full additive subcategory of $\mmod A$ generated by the direct summands of $S_e$ and which is closed under isomorphisms. Observe that for $S$ indecomposable in add$(S_e)$, $G(S)$ is a (graded) finitely generated projective $Y(e)$-module generated in degree $0$.

\section{Projective covers and approximations}

In this section, we will relate projective covers in $\mmod Y(e)$ with some minimal left-approximations in $D(A)$. We start with the following, which is well known by the specialists.

\begin{Lemma} \label{Lemma0} Let $P,Q$ be projective in $\mmod A$ and $f: P \to Q$ a morphism which is not radical. Then there are decompositions $Q = Q_1 \oplus Q_2$ and $P=P_1 \oplus P_2$ such that, under these decompositions, $f = f_1 \oplus f_2$ where $f_1$ is an isomorphism and $f_2$ is radical.
\end{Lemma}

\begin{proof} We proceed by induction on $s$, the number of indecomposable direct summands of $P$, which is well defined by the Krull-Schmidt property. There is a direct summand $Q'$ of $Q$ such that $\pi f$ is surjective, where $\pi: Q \to Q'$ is the canonical projection. Now, since $\pi f$ splits, there is a direct summand $P'$ of $P$ with $P = P' \oplus {\rm ker}(\pi f)$ and where the restriction of $\pi f$ to $P'$ is an isomorphism. Let $Q_1=f(P')$, which is also a direct summand of $Q$ and let $\pi_1: Q \to Q_1$ be the canonical projection. We have $Q = Q_1 \oplus Q_2$ and $P = P_1 \oplus P_2$ where $P_2 = {\rm ker} (\pi_1 f)$ and $P_1 = P'$. Observe that under these decompositions, $f = f_{P_1} \oplus f_{P_2}$ where $f_{P_i}$ denotes the restriction of $f$ to $P_i$. If $f_{P_2}: P_2 \to Q_2$ is radical, then we are done. In particular, this settles the case where $s=1$. Otherwise, since $\mmod A$ is Krull-Schmidt, we proceed by induction on $f_{P_2}$.
\end{proof}

Denote by $\mathcal{S}_e$ the full additive subcategory of $D(A)$ generated by the shifts of the objects in add$(S_e)$.  Recall that $\mathcal{S}_e$ is said to be \emph{covariantly finite} in $D(A)$ if for any $M \in D(A)$, there is a morphism $f_S: M \to S$ with $S \in \mathcal{S}_e$ such that $\Hom_{D(A)}(f_S,S')$ is surjective whenever $S' \in \mathcal{S}_e$. We start with the following easy lemma.

\begin{Lemma} If $A$ has finite global dimension, then $\mathcal{S}_e$ is covariantly finite in $D(A)$.
\end{Lemma}

\begin{proof} Let $M \in D(A)$. Consider a complex $P^\bullet \in D(A)$
$$P^\bullet: \qquad \cdots \to P^{-2} \to P^{-1} \to P^0 \to P^1 \to \cdots$$
of finitely generated projective $A$-modules such that $M \cong P^\bullet$. By Lemma \ref{Lemma0}, we may assume that the complex has radical differentials.  Since $A$ has finite global dimension, $P^\bullet$ is a bounded complex. We let $S_i$ denote the top of $P_{-i}$. It is clear that $M \to \textstyle{\bigoplus}_{j \in \Z}S_j[j]$ is a left $\mathcal{S}_e$-approximation of $M$. It need not be minimal.
\end{proof}

Note that since $\mathcal{S}_e$ is clearly a Krull-Schmidt category, any left $\mathcal{S}_e$-approximation $f: M \to S$ of an object $M \in D(A)$ can be made minimal: there is a direct summand $S'$ of $S$ such that the co-restriction $f'$ of $f$ to $S'$ is again a left $\mathcal{S}_e$-approximation of $M$ and any morphism $h: S' \to S'$ with $hf' = f'$ has to be an isomorphism. For a reference, the reader may refer to Corollary 1.4 in \cite{KrSao}.

\medskip

\begin{Lemma} \label{Lemma1} Let $M \in D(A)$, $i \in \Z$ and $S \in {\rm add}(S_e[i])$. Then a graded morphism $h: G(S) \to G(M)$ is uniquely determined by $G(S)(i) \to G(M)(i)$ and in particular, there is a unique $f: M \to S$ such that $h = G(f)$.
\end{Lemma}

\begin{proof}
We may assume that $S$ is indecomposable and that $i=0$. Thus, we may assume that $S$ is a direct summand of $S_{e}$. Then, $G(S)$ is indecomposable projective generated in degree $0$. The first part is clear since $h$ is graded. For the second part, since $h$ is uniquely determined by $h_0: \Hom(S,S_e) \to \Hom(M,S_e)$, we may take $f = \pi \circ h_0(\iota_{S})$ where $\iota_S: S \to S_e$ is the split inclusion and $\pi: S_e \to S$ is the projection.
\end{proof}

Since $Y(e)$ is $\mathbb{Z}$-graded (with the negative homogeneous pieces being all zero) and $Y(e)(0)$ is Artinian, it follows from \cite{Dasc} that $Y(e)$ is graded semiperfect and, in particular, the finitely generated $Y(e)$-modules admit graded projective covers.

\begin{Lemma} \label{LemmaMinimal}
Let $M \in D(A)$ and $f: M \to S$ be a minimal left $\mathcal{S}_e$-approximation of $M$ in $D(A)$. Then $G(f): G(S) \to G(M)$ is a graded projective cover in $\mmod Y(e)$.
\end{Lemma}

\begin{proof}
By the approximation property, $\Hom(f, S_e[i])$ is an epimorphism for all integers $i$. This means that $G(f)(i) = \Hom(f, S_e[i]): G(S)(i) \to G(M)(i)$ is an epimorphism for all $i$, and in particular, $G(f)$ is an epimorphism in $\mmod Y(e)$. Assume that $G(f)$ is not a graded projective cover. Since $\mmod Y(e)$ is Krull-Schmidt, it means that $G(S) = P_1 \oplus P_2$ in $\mmod Y(e)$ where the restriction of $G(f)$ to $P_2$ is zero while $P_2$ is non-zero. There exist $S_1, S_2 \in \mathcal{S}_e$ with $P_i = G(S_i)$ and $S_2 \ne 0$. By Lemma \ref{Lemma1}, there are unique $f_i: M \to S_i$ such that $G(f) = G ([f_1, f_2]^T) = [G(f_1), G(f_2)]^T$. In particular, $f_2 = 0$. By uniqueness of $f$, we have $f = [f_1, 0]^T$, which implies that $f$ is not minimal, a contradiction.
\end{proof}

\section{Homological relationships between $A, \Gamma$ and $Y(e)$}

In this section, we investigate the relationships between the global dimensions of $A, \Gamma$ and $Y(e)$. We start with our first result.

\begin{Prop} \label{Prop0}
Suppose that $A$ and $Y(e)$ have finite global dimension. Then $\Gamma$ has finite global dimension.
\end{Prop}

\begin{proof}
Let $T$ be a simple $A$-module not in add$(S_e)$ and take $S_1 \in \mathcal{S}_e$ with a minimal left $\mathcal{S}_e$-approximation $f_0: T \to S_1$. Consider the cone $C_1$ of $f_0$. Applying $F = R\Hom((1-e)A,-)$ to the exact triangle
$$T \to S_1 \to C_1 \to T[1],$$
and using the fact that $F$ vanishes on $\mathcal{S}_e$, we get that $F(T)$ is isomorphic to $F(C_1)[-1]$. Since $f_0$ is a left $\mathcal{S}_e$-approximation of $T$, we see that $\Hom_{D(A)}(f_0,S_e[i])$ is an epimorphism for all $i$. Therefore, by applying the functor $\Hom_{D(A)}(-, S_e)$, we have a short exact sequence
$$0 \to \Hom_{D(A)}(C_1,S_e[i]) \to \Hom_{D(A)}(S_1,S_e[i]) \to \Hom_{D(A)}(T,S_e[i]) \to 0.$$
Therefore, applying $G = \textstyle{\bigoplus}_{i \in \Z}\Hom_{D(A)}(-, S_e[i])$, we get a short exact sequence
$$0 \to G(C_1) \to G(S_1) \to G(T) \to 0$$
in $\mmod Y(e)$ where the rightmost morphism is a projective cover, by Lemma \ref{LemmaMinimal}. In general, for $i \ge 1$, take $S_i \in \mathcal{S}_e$ with a minimal left $\mathcal{S}_e$-approximation $f_{i-1}: C_{i-1} \to S_{i}$. Consider the cone $C_i$ of $f_{i-1}$. Then $F(C_i)[-i]$ is isomorphic to $F(T)$ and $G(C_i)$ is the $i$-th syzygy of $G(T)$. Since $Y(e)$ has finite global dimension, there is some $r \ge 1$ such that $G(C_r)=0$. This means that there is a complex of projective modules $P^\bullet$, quasi-isomorphic to $C_r$, such that no direct summand of a term of $P^\bullet$ lies in ${\rm add}((1-e)A)$. Since $F(C_r)[-r]$ is isomorphic to $F(T)$, we see that $F(C_r)[-r]$ is a projective resolution of $F(T)$. Since $T$ and all objects in $\mathcal{S}_e$ have bounded projective resolutions in $D^b(\mmod A)$, we see that $P^\bullet$ is bounded and hence, $F(C_r)[-r]$ is a finite projective resolution. Since all simple modules in $\mmod \Gamma$ are of the form $F(T)$ with $T$ a simple $A$-module not in ${\rm add}(S_e)$, we see that ${\rm gl.dim}\, \Gamma$ is finite by \cite[Prop. 2.2]{IP}.
\end{proof}

\begin{Remark}
In the above proposition, let $e$ be primitive. Since $A$ has finite global dimension, we see that $Y(e)$ is Artinian. Since $Y(e)$ is local, the only way it can have finite global dimension is when it is a simple ring, that is, $\Ext^i_A(S_e, S_e)=0$ for  all $i > 0$. Therefore, for $e$ primitive, the statement reads as: If $A$ has finite global dimension and $\Ext^i_A(S_e, S_e)=0$ for all $i > 0$, then $\Gamma$ has finite global dimension. This is Proposition 4.4 (3) in \cite{IP}.
\end{Remark}

We continue our investigation with the following.

\begin{Prop} \label{Prop1}
Suppose that $\Gamma$ has finite global dimension and $Y(e)$ is Artinian. Then $A$ has finite global dimension.
\end{Prop}

\begin{proof}
Let $T$ be a simple $A$-module. Denote by $\Omega_j$ its $j$-th syzygy. Assume first that $T$ lies in add$(S_e)$. Since $Y(e)$ is Artinian, all but finitely many $Y(e)(i)$ are zero. Therefore, there is some $j \ge 0$ such that $\Omega_j$ satisfies $\Ext^i_A(\Omega_j, S_e)=0$ for all $i \ge 0$. This means that a minimal projective resolution of $\Omega_j$ has all of its terms in add$((1-e)A)$. Then, the projective dimension of $F(\Omega_j)$, which is finite since ${\rm gl.dim}\, \Gamma < \infty$, coincides with the projective dimension of $\Omega_j$. Now, the projective dimension of $T$ is $j$ plus the projective dimension of $\Omega_j$, and thus is finite. Since $A$ is Noetherian, a finitely generated module is projective if and only if it is flat. Therefore, $S_e$ has finite flat dimension. Now, observe that $\Gamma$ is left and right Noetherian. Therefore, we have ${\rm gl.dim} \Gamma^{\rm \, op} = {\rm gl.dim} \Gamma$. Therefore, using a similar argument as above, we get that $Ae/{\rm rad A}\,e$ has finite flat dimension in $\mmod A^{\, \rm op}$. In particular, there exists some $m > 0$ with ${\rm Tor}^i_A(S_e, -)=0, {\rm Tor}^i_A(-, Ae/{{\rm rad}A}\,e)=0$ whenever $i \ge m$.
Let now consider the case where $T$ is simple not in add$(S_e)$. Since ${\rm Tor}^i_A(T, Ae/{{\rm rad}A}\,e)=0$ for $i \ge m$, it means that no direct summand of $eA$ appears in a minimal projective resolution of $\Omega_m$. As argued above, this implies that $T$ has finite projective dimension. Since $A$ is Noetherian semiperfect, the global dimension of $A$ is the supremum of the projective dimensions of the simple modules in $\mmod A$; see \cite[Prop. 2.2]{IP}. The statement follows.
\end{proof}

\begin{Remark} In the previous proposition, one cannot replace the condition "$Y(e)$ is Artinian" by the condition "$Y(e)$ has finite global dimension". For instance, consider the Nakayama algebra $A$ of rank $3$ over a field $k$ with a vanishing radical squared. In other words, $A$ is the Koszul algebra of the oriented cycle of length $3$ with all possible quadratic relations. Take $e$ to be any primitive idempotent. Then $\Gamma$ is hereditary and $Y(e)$ is a polynomial algebra over $k$. Indeed, $Y(e)$ is an idempotent subalgebra of the quadratic dual $Y(1) = A^!$ of $A$, where $A^!$ is the Koszul algebra of the oriented cycle of length $3$ without relations. In particular, both $\Gamma, Y(e)$ have finite global dimensions. However, $A$ has infinite global dimension.
\end{Remark}

The next result follows directly from propositions \ref{Prop0} and \ref{Prop1}.

\begin{Cor} Assume that $Y(e)$ is Artinian and has finite global dimension. Then $A$ has finite global dimension if and only if $\Gamma$ has finite global dimension.
\end{Cor}

The following result is surprising.

\begin{Cor} Let $A$ be finite dimensional over a field $k = \bar k$ and assume that $e$ is primitive. Assume further that ${\rm gl.dim}\, \Gamma < \infty$. Then either $\Ext^i_A(S_e, S_e)$ is non-zero for infinitely many $i$, or else it vanishes for all positive $i$.
\end{Cor}

\begin{proof}Assume that $\Ext^i_A(S_e, S_e)$ vanishes for $i$ sufficiently large. Then $Y(e)$ is finite dimensional. It follows from Proposition \ref{Prop1} that $A$ has finite global dimension. It follows form the validity of the Strong no loop conjecture in this setting, see \cite{ILP}, that $\Ext^1_A(S_e, S_e)=0$. Now, it follows from \cite[Theo. 6.5]{IP} that $\Ext^i_A(S_e, S_e)=0$ for all $i > 0$.
\end{proof}

The next step in our investigation would be to assume the finiteness of the global dimensions of $A, \Gamma$ and see if we can get some homological properties for $Y(e)$. In general, the fact that both $A, \Gamma$ have finite global dimension does not imply that $Y(e)$ has finite global dimension. Indeed, by taking the extreme case $e=1$, we reduce to the question of whether $A$ being of finite global dimension implies that the Yoneda ring $Y(1)$ of $A$ is of finite global dimension. This is not true and it is easy to find counter-examples of this statement.
 The other extreme case is when $e$ is primitive.
In this case, it was proven in \cite{IP} that if $A$ is a $k$-algebra over a field $k$, $A/{\rm rad}A$ is finite dimensional and $\Ext^1_A(S_e, S_e)=0$, then $A, \Gamma$ both have finite global dimensions imply that $\Ext^1_A(S_e, S_e)=0$ for all $i > 0$. In particular, $Y(e)$ has finite global dimension (it is a one dimensional $k$-algebra). Therefore, a condition on $e$ seems necessary.

\medskip

The \emph{Ext-quiver} $Q_e$ of $e$ is obtained as follows. Decompose $e$ into a sum of ${\rm rk}(e)$ pairwise orthogonal primitive idempotents, say $e = e_1 + \cdots + e_{{\rm rk}(e)}$. The vertices of $Q_e$ are $e_1, \ldots, e_{{\rm rk}(e)}$. For $1 \le i,j \le {\rm rk}(e)$, we put an arrow between $e_i$ and $e_j$ if $\Ext^1_A(S_{e_i}, S_{e_j}) \ne 0$. Observe that if $A=kQ/I$ where $Q$ is a finite quiver, $k$ is a field and $I$ is an admissible ideal of $kQ$, then $Q_e$ is the full subquiver of $Q$ corresponding to the vertices $e_1, e_2, \ldots, e_{{\rm rk}(e)}$, where parallel arrows are identified.
We will call $e$ \emph{acyclic} if $Q_e$ has no oriented cycles. Observe that if $A=kQ/I$ is as above and $e$ is primitive and the projective dimension of $S_e$ is finite (or if the global dimension of $A$ is finite), then $e$ is acyclic. This indeed follows from the validity of the Strong no loop conjecture; see \cite{ILP}. We make the following conjecture.

\medskip

\begin{Conj} \label{conj} Assume that $e$ is acyclic. If both $A, \Gamma$ have finite global dimensions, then so does $Y(e)$.
\end{Conj}

As observed above, this conjecture holds if $e$ is primitive, $A$ is an algebra over a field $k$ and $A/{\rm rad}A$ is finite dimensional. In the next section, we will see that the conjecture holds true when $A$ is finite dimensional positively graded with $S_e$ having a linear projective resolution. In particular, the conjecture holds true when $A$ is finite dimensional Koszul. Moreover, acyclicity of $e$ is not needed in this case.

\medskip

Let us denote by $\T$ the Serre subcategory of $\mmod A$ generated by the objects in add$(S_e)$. In other words, $\T$ is the full subcategory of the modules $M$ with $\Hom_A((1-e)A, M)=0$. The following holds in general but will be used in the next section to establish the conjecture when $A$ is Koszul.

\begin{Lemma} \label{LemmaSplit}
Let $f:Q \to P$ be a morphism in $\mmod A$ with $P$ projective in ${\rm add}((1-e)A)$ such that $F(f) = \Hom_A((1-e)A,f)$ is a section. Then $Q = Q_1 \oplus Q_2$ where $Q_1$ is projective in ${\rm add}((1-e)A)$ and $Q_2 \in \T$.
\end{Lemma}

\begin{proof}
Let $K$ be the kernel of $f$. Since $F(f)$ is a monomorphism, $F(K)=0$, so $K \in \T$. Let $E$ be the set of the submodules $X$ of $Q$ with $Q/X \in \mathcal{T}$. Let $Q':=\textstyle{\bigcap}_{X \in E}X$. It is straightforward to check that $S:=Q/Q'$ is the maximal quotient of $Q$ lying in $\T$. Thus, we have a short exact sequence $$0 \to Q' \stackrel{h}{\to} Q \to S \to 0$$
with $\Hom(Q', S_e)=0$. Moreover, $Q',S$ are finitely generated since $A$ is Noetherian. We have $F(fh)=F(f)F(h)$ where $F(h)$ is an isomorphism. Thus, $F(fh)$ is a section. Let $t: Q'/{\rm rad} Q' \to P/{\rm rad}P$ be the morphism induced from $fh$ on the respective tops of $Q', P$. Observe that $F(Q'/{\rm rad} Q')$ is a direct summand of $F(Q')/{\rm rad}F(Q')$ and similarly, $F(P/{\rm rad} P)$ is a direct summand of $F(P)/{\rm rad}F(P)$. Since $F(fh)$ is a section, it induces an injective morphism $F(Q')/{\rm rad}F(Q') \to F(P)/{\rm rad}F(P)$. This implies that $F(t)$ is injective and hence, that $t$ is injective. Now, the image of $t$ defines a direct summand $R$ of $P$ such that the co-restriction $f': Q \to R$ of $f$ to $R$ is such that $f'h$ is an isomorphism. Therefore, $h$ is a section, so $Q \cong Q' \oplus S$ with $Q'$ finitely generated projective.
\end{proof}

\begin{Remark}[Recollements] It is well known that the idempotent $1-e$ induces a recollement of ${\rm Mod}\,A$ by ${\rm Mod}\, \Gamma$ and ${\rm Mod}\, A/A(1-e)A$. However, in general, the three rings $A, \Gamma, A/A(1-e)A$ are very different in their homological aspects. In particular, the finiteness of the global dimension of $A$ does not imply the finiteness of the global dimensions of the rings $\Gamma, A/A(1-e)A$. There is a slightly different situation when considering derived categories. Consider $\mathcal{X}$ the smallest full triangulated subcategory of the (unbounded) derived category $\mathcal{D}(A)$ of $A$ containing $(1-e)A$ and that is closed under small coproducts. It is not hard to see that $F$ induces a triangle-equivalence between $\mathcal{X}$ and the (unbounded) derived category $\mathcal{D}(\Gamma)$ of $\Gamma$. Now, set $\mathcal{W}$ the full subcategory of $\mathcal{D}(A)$ of objects $W$ with $\Hom_{\mathcal{D}(A)}(X,W)=0$ for all $X \in \mathcal{X}$. Clearly, $\mathcal{W}$ is a triangulated subcategory of $\mathcal{D}(A)$ and coincides with the full triangulated subcategory of $\mathcal{D}(A)$ of the objects having cohomologies annihilated by $(1-e)$. Consider the dg-ring $B:=R\Hom(S_e, S_e)$. It follows from \cite[Prop. 3.4]{BP} that there is a recollement of $\mathcal{D}(A)$ by $D(B)$ and $\mathcal{D}(\Gamma)$. Again, the finiteness of the global dimension of $A$ does not seem to imply any nice (obvious) homological behavior for $D(B)$ and $\mathcal{D}(\Gamma)$. Notice, though, that the cohomology ring of $B$ is precisely $Y(e)$.
\end{Remark}

\section{Positively graded finite dimensional algebras}

Let $k$ be a field. In this section, we assume that $A$ is a finite dimensional $k$-algebra that is positively graded, that is, we have $$A = A(0) \oplus A(1) \oplus \cdots$$ as a $k$-vector space, $A(0)$ is a finite product of copies of $k$ and for $i, j \ge 0$, we have $A(i) A(j) \subseteq A(i+j)$. All $A$-modules considered will be graded right $A$-modules. Given a graded module $M = \textstyle{\bigoplus}_{j \in \Z} M(j)$, we let $M\langle i \rangle$ denote the graded module $M'$ with $M'(j) = M(i+j)$.

\medskip

Recall that $A$ is \emph{Koszul} if (it is positively graded and) the semi-simple $A$-module $A(0)$ has a linear (graded) projective resolution. In other words, there is a graded projective resolution
$$\cdots \to P_1 \to P_0 \to A(0) \to 0$$
of $A(0)$ such that for $i \ge 0$, the projective module $P_i$ is finitely generated in degree $i$, so is a (graded) direct summand of a direct sum of copies of $A\langle -i \rangle$. In this case, the given projective resolution has to be minimal, since the morphisms of the projective resolution are all radical.
Observe that $S_e$ is a direct summand of $A(0)$. In this section, we prove Conjecture \ref{conj} in case $S_e$ has a linear projective resolution. Observe that if $S_e$ has a linear projective resolution, then any $S \in \mathcal{S}_e$ may be given by a complex of projective modules with all morphisms linear. Indeed, let $S \in \mathcal{S}_e$, so $S$ is isomorphic to $S_1 \oplus \cdots \oplus S_r$ where each $S_i$ is a shift (as a complex) by $t_i$ of an indecomposable simple $A$-module in add$(S_e)$. Then we may replace $S$ by $S_1\langle -t_1 \rangle \oplus \cdots \oplus S_r\langle -t_r \rangle$, since the graded shift does not change the underlying module. It is now clear that $S$ can be given by a complex of finitely generated projective modules
$$\cdots \to Q^{-2} \to Q^{-1} \to Q^0 \to \cdots$$
that is bounded above and where $Q^{-i}$ is generated in degree $i$ for all $i$. Using this fact, we get the following.

\begin{Lemma} \label{HomDegreeZero}Let $S_1, S_2 \in \mathcal{S}_e$ be represented by linear complexes of projective modules and $f: S_1 \to S_2$. Then $f=(f_i)_{i \in \Z}$ can be chosen, up to homotopy, so that all $f_i$ are homogeneous of degree zero.
\end{Lemma}

\begin{proof} It suffices to check it for $S_1$ indecomposable in $\mmod A$ and $S_2$ a non-negative shift, say by $t$, of an indecomposable $S$ in $\mmod A$. Let
$$\cdots \to Q^{-2} \to Q^{-1} \to Q^0 \to 0 \to \cdots$$
be a linear graded projective resolution of $S_1$ and
$$\cdots \to R^{-t-2} \to R^{-t-1} \to R^{-t} \to 0 \to  \cdots$$ a linear graded projective resolution of $S$. All $f_i: Q^i \to R^i$ are zero for $i > -t$ and $f_{-t}$ is a retraction, hence can be chosen homogeneous of degree zero. Now, the lifts $f_i$ of $f_{-t}$ for $i \le -t-1$ can all be chosen to be homogeneous of degree zero by working in the category of graded $A$-modules.
\end{proof}

The following can be checked directly using the definition of left $\mathcal{S}_e$-approximations.

\begin{Lemma} \label{Lemma2}Let $C \in D(A)$ with a minimal left $\mathcal{S}_e$-approximation $C \to S$ and consider the induced exact triangle $C[-1] \stackrel{u}{\to} S[-1] \stackrel{v}{\to} C' \to C$. Then ${\rm ker}\,\Hom(u,-)\mid_{\mathcal{S}_e} \cong \Hom(C',-)\mid_{\mathcal{S}_e}$.
\end{Lemma}

Let $C_0, \C_1, \ldots, C_r$ be complexes of finitely generated projective $A$-modules
$$C_i: \cdots \to P_i^{j-1} \stackrel{d_i^{j-1}}{\to} P_i^{j} \stackrel{d_i^{j}}{\to} P_i^{j+1} \to \cdots$$
with morphisms $f_i=(f_i^{j}:P_i^{j} \to P_{i+1}^{j}): C_i \to C_{i+1}$ for $0 \le i \le r-1$. Assume that $f_{i+1}^{j}f_i^{j}=0$ for all $0 \le i \le r-2$ and all $j$. Therefore, we have a double complex $DC(C_0, \ldots, C_r)$ of projective modules and we may consider the total complex $T(C_0, \ldots, C_r)$ given by
$$\cdots \to \textstyle{\bigoplus}_{i+j=1}P_i^{j} \to \textstyle{\bigoplus}_{i+j=2}P_i^{j} \to \cdots$$
where the differential $d^l: \textstyle{\bigoplus}_{i+j=l}P_i^{j} \to \textstyle{\bigoplus}_{i+j=l+1}P_i^{j}$ is the usual differential such that its restriction to $P_i^{l-i}$ is given by $(-1)^id_i^{l-i}+f_i^{l-i}$. The following lemma is easy to check and left to the reader.

\begin{Lemma} \label{LemmaTotal}Using the above notation, assume that we have a morphism of complexes $f_r: C_r \to C_{r+1}$ such that $f_{r}^{j}f_{r-1}^{j}=0$ for all $j$. Then $f_r$ induces a morphism of complexes $f: T(C_0, \ldots, C_r) \to C_{r+1}[-r]$ in a canonical way: if $x = (x_0, \ldots, x_r) \in \textstyle{\bigoplus}_{i+t=j}P_i^{t}$, then $f^j(x) = f_r^{-r+j}(x_r)$. Moreover, the mapping cone of $f$ coincides with $T(C_0, \ldots, C_r, C_{r+1})[1]$.
\end{Lemma}

The following lemma is crucial for computing cohomology of a total complex as above. We need the following notation. For $U$ a direct summand of $V$ in $D(A)$, we denote by $\iota_{U}: U \to V$ the canonical injection and by $\pi_{U}: V \to U$ the canonical projection.

\begin{Lemma} \label{Lemma3}Let $M_0 \in D(A)$ be a graded $A$-module 
and assume that $M_0, S_e$ admit linear projective resolutions. For $i \ge 0$, let $f_i: M_i \to S_i$ be a minimal left $\mathcal{S}_e$-approximation of $M_i$ with an exact triangle $M_{i+1} \to M_i \to S_i \to M_{i+1}[1]$. For each $j \le -1$, we have a short exact sequence
$$0 \to H^{j-1}(S_i) \to H^j(M_{i+1}) \to H^j(M_i) \to 0$$
in cohomology.
\end{Lemma}

\begin{proof}
Let $g_i:S_i[-1] \to M_{i+1}, t_{i+1}: M_{i+1} \to M_i$ be such that we have an exact triangle $$S_i[-1] \stackrel{g_i}{\to} M_{i+1} \stackrel{t_{i+1}}{\to} M_i \stackrel{f_i}{\to} S_i$$ and set $h_i:=(f_{i+1}g_i)[i+1]: S_i[i] \to S_{i+1}[i+1]$. We replace the objects $M_0, S_0, S_1, \ldots$ by complexes of finitely generated projective $A$-modules where all differentials are linear. We claim that for $i \ge 1$, $M_i$ is quasi-isomorphic to the total complex $$T(M_0, S_0, S_1[1], S_2[2], \ldots,S_{i-1}[i-1])$$ with morphisms $f_0: M_0 \to S_0, h_0: S_0 \to S_1[1], h_1: S_1[1] \to S_2[2], \ldots, h_{i-2}: S_{i-2}[i-2] \to S_{i-1}[i-1]$ each of which is made from morphisms of projective modules homogeneous of degree zero. We prove this by induction on $i \ge 1$. For convenience, we write $h_{-1}=f_0$.  
For $i=1$, the claim follows from Lemma \ref{HomDegreeZero} and the definition of the mapping cone of $f_0$. Assume that the claim holds for some $i \ge 1$. It follows from Lemma \ref{Lemma2} (by taking $C = M_{i-1}, C' = M_i$ and $S = S_{i-1}$) that $f_i: M_i \to S_i$ is given by the data of a morphism $u_i:S_{i-1}[-1] \to S_{i}$ whose composition with $f_{i-1}[-1]$ is zero (and then, $u_i = f_ig_{i-1}$). Let $u_i$ be such a morphism. By Lemma \ref{HomDegreeZero}, $u_i$ is given by morphisms of degree zero between projective modules. Set $h_{i-1} = u_i[i] = f_{i}g_{i-1}[i]$. We see that $$h_{i-1}h_{i-2} = (f_{i}g_{i-1})[i](f_{i-1}g_{i-2})[i-1] = f_{i}[i](g_{i-1}[i]f_{i-1}[i-1])g_{i-2}[i-1]=0.$$ Since the
$h_{i-2}^j, h_{i-1}^j$ are all homogeneous of degree zero and since the differentials in $S_{i-2}, S_i$ are all radical, we see that the morphisms $h_{i-2}^jh_{i-1}^j$ are all zero. Indeed, if a morphism of degree zero between finitely generated projective modules generated in the same degree lies in the radical, then it has to be the zero morphism. It follows from this and Lemma \ref{LemmaTotal} that $M_{i+1}$, which is the shift by $-1$ of the mapping cone of $f_i$, is of the required form.

To prove the statement on cohomology, it suffices to prove that $H^j(f_i)=0$ for all $i \ge 0$ and all $j \le -1$. We proceed by induction on $i$. The statement is clear for $i=0$ since $H^t(M_0)=0$ for all $t \le -1$. Let $i$ be positive. We need to prove that $H^j(f_i)=0$ for all $j \le -1$. We first claim that $H^j(f_ig_{i-1})=0$ for all $j \le -1$. Assume otherwise, that is, there is $j \le -1$ with $H^j(f_ig_{i-1})\ne 0$. Let $A$ be a maximal direct summand of $S_{i-1}[-1]$ in add$(S_e[j])$ and $B$ be a maximal direct summand of $S_{i}$ in add$(S_e[j])$. We have a morphism $\pi_Bf_ig_{i-1}\iota_A: A \to B$ with $H^j(\pi_Bf_ig_{i-1}\iota_A) \ne 0$, or equivalently, $H^0((\pi_Bf_ig_{i-1}\iota_A)[-j]) \ne 0$. Since $A[-j], B[-j]$ are semi-simple, it means that there are indecomposable direct summands $A'[-j]$ and $B'[-j]$ of $A[-j]$ and $B[-j]$, respectively, such that $\pi_{B'[-j]}(\pi_Bf_ig_{i-1}\iota_A)[-j]\iota_{A'[-j]}$ is an isomorphism. This implies that the restriction of the morphism $g_{i-1}: S_{i-1} \to M_i[1]$ to $A'$ is a section, which means that the co-restriction of $f_i$ to $A'$ is zero. This contradicts that $f_i$ is minimal. This proves our claim. Now, by induction, we have a short exact sequence
$$0 \to H^j(S_{i-1}[-1]) \stackrel{H^j(g_{i-1})}{\to} H^j(M_i) \stackrel{H^j(t_i)}{\to} H^j(M_{i-1}) \to 0.$$
Since $H^j(f_ig_{i-1}) = H^j(f_i)H^j(g_{i-1})=0$, there is a morphism $w: H^j(M_{i-1}) \to H^j(S_i)$ such that $H^j(f_i)=wH^j(t_i)$. However, $w$ is clearly zero as any element in degree $j$ of $T(M_0, S_0, S_1[1], S_2[2], \ldots,S_{i-2}[i-2])$, seen canonically as an element in degree $j$ of $T(M_0, S_0, S_1[1], S_2[2], \ldots,S_{i-1}[i-1])$, vanishes after applying $f_i$ to it.

\end{proof}

\begin{Theo} \label{Theo1}
Let $A$ be positively graded finite dimensional and assume that $S_e$ has a linear projective resolution. If the global dimensions of $A$ and $\Gamma$ are finite, then the global dimension of $Y(e)$ is finite.
\end{Theo}

\begin{proof}
Let $T$ be a simple $A$-module in add$(S_e)$ with syzygy $M_0$. For $i \ge 0$, let $S_i \in \mathcal{S}_e$ with a minimal left $\mathcal{S}_e$-approximation $f_i: M_i \to S_i$ with mapping cone $M_{i+1}[1]$. Applying $G$ to the exact triangle
$$M_{i+1} \to M_i \to S_i \to M_{i+1}[1],$$
we get a short exact sequence
$$0 \to G(M_{i+1}[1]) \to G(S_i) \to G(M_i) \to 0$$
in $\mmod Y(e)$ where the rightmost morphism is a projective cover, since $f_i$ is minimal. Observe also that all $F(M_i)$ are isomorphic to $F(M_0)$ and the $M_i$ are all concentrated in non-positive degrees. Also, it follows from the construction of the $M_i$ that for $j \le -1$, we have $H^j(M_i) \in \T$. Let $s$ be the projective dimension of the $\Gamma$-module $F(M_0)$. We need to prove that $G(M_i)=0$ for some $i \ge 0$. Equivalently, we need to prove that if $i$ is large enough and $M_i$ is given by a complex of projective modules with radical maps, then no term has a non-zero direct summand in add$(eA)$.

Observe that for $i \ge 1$, using the notations of the above lemma, we have an exact triangle
$$M_i \to M_0 \to T_i \to M_i[1]$$
where $T_i:=T(S_0, S_1[1], \ldots, S_{i-1}[i-1])$
and we have an epimorphism $H^0(M_0) \to H^0(T_i)$, as $H^1(M_i)=0$. Observe also that the latter epimorphism factors through $H^0(T_{i+1}) \to H^0(T_i)$. Therefore, we have a chain
$$\cdots \to H^0(T_3) \to H^0(T_2) \to H^0(T_1)$$ of epimorphisms and since $H^0(M_0) \cong M_0$ is finite dimensional,
we see that there exists $i_0 \ge 1$ with $H^0(T_{i_0}) \cong H^0(T_i)$ whenever $i \ge i_0$. This means that $S_i[i]$ has no non-zero direct summand in add$(S_e[i])$ for $i \ge i_0$. Equivalently, for $i \ge i_0$, we have $\Hom_{D(A)}(M_i, S_e)= 0$. Now, let $i$ be a non-negative integer and $j$ be minimal such that $\Hom_{D(A)}(M_i, S_e[j])\ne 0$. We claim that there is $r > 0$ such that $\Hom_{D(A)}(M_{i+r}, S_e[j']) = 0$ for $j' \le j$. The case where $j=0$ has just been settled, so we may assume that $j \ge 1$ (so $i \ge i_0$). In the minimal left $\mathcal{S}_e$-approximation $f_i: M_i \to S_i$ of $M_i$, there is a non-zero module $Z$ in add$(S_e)$ such that $Z[j]$ is a direct summand of $S_i$. Observe that we have a short exact sequence
$$0 \to H^j(S_i) \to H^{j+1}(M_{i+1}) \to H^{j+1}(M_i) \to 0$$
for $j \le -1$. Indeed, this follows from Lemma \ref{Lemma3} for $j \le -2$ and from the fact that $H^0(S_i)=0$, by what we have just proven, for $j=-1$.
Therefore, the dimension of $H^{j+1}(M_{i+1})$ is larger than that of $H^{j+1}(M_{i})$. Now, $M_{i+1}$ has the property that $\Hom_{D(A)}(M_{i+1}, S_e[j']) = 0$ if $j' < j$. Therefore, if we continue in this way, we see that $$\cdots \ge {\rm dim}_k H^{j+1}(M_{i+2}) \ge {\rm dim}_k H^{j+1}(M_{i+1}) > {\rm dim}_k H^{j+1}(M_{i})$$
with $\Hom_{D(A)}(M_{i+i'}, S_e[j']) = 0$ whenever $i' \ge 0$ and $j' < j$.
Since the $j+1$-th and $j+2$-th terms of $M_i, M_{i+1}, \ldots$ are all the same and finite dimensional, there is a bound on the dimensions of $H^{j+1}M_{i}, H^{j+1}M_{i+1}, \ldots$. Therefore, there is $r>0$ such that $H^{j+1}(M_{i+r}) \cong H^{j+1}(M_{i+r-1})$ which, by Lemma \ref{Lemma3}, implies that $H^j(S_{i+r})=0$ or, equivalently, that $\Hom_{D(A)}(M_{i+r}, S_e[j]) = 0$. The claim follows from this.

Assume to the contrary that $G(M_i)\ne 0$ for all $i \ge 0$. It follows from the claim and Lemma \ref{Lemma3} that there are $t,r > 0$ such that $\Hom_{D(A)}(M_t, S_e[j]) = 0$ for $0 \le j \le s+r$ and $H^{-s-r+1}(M_t)$ is non-zero. Set $q = -s-r+1$. Assume that $M_t$ is given by a complex
$$\cdots \to P^{q-1} \stackrel{d^{q-1}}{\to} P^q \stackrel{d^{q}}{\to} P^{q+1} \to \cdots \to P^0 \to 0$$
of projective modules where the differentials are radical.
Let $h:K \to P^q$ be the kernel of $d^q$. Observe that $\Hom_A(P^{q-1}, S_e) = 0$. Let $K'$ be a submodule of $K$ with $K/K' \cong H^q(M_t) \ne 0$.
Observe that we have the beginning
$$ F(P^q) \to F(P^{q+1}) \to \cdots \to F(P^0) \to 0$$ of a projective resolution of $F(M_0)$ and the kernel of the last morphism is $F(h)$. Since $s < s+r$ is the projective dimension of $F(M_0)$, we see that $F(h)$ is a section.
By Lemma \ref{LemmaSplit}, $K = K_1 \oplus K_2$ where $K_1$ is projective in $\mmod A$ and $K_2 \in \T$. Moreover, the restriction $K_1 \to P^q$ is a section. Since we have assumed that all maps in $M_t$ are radical, we get $K_1 = 0$. Thus, $K=K_2$ and $K' = 0$ (hence, $d^{q-1}=0$) since $\Hom_A(P^{q-1}, S_e) = 0$. Since $H^{q-1}(M_t) \in \T$, this leads to $P^{q-1}=0$. If all $P^i$ for $i \le q-1$ are zero, we get a contradiction. Otherwise, there is some $p \le q-2$ with $P^p \ne 0$ and $\Hom_A(P^p, S_e) \ne 0$. Then, considering the minimal $\mathcal{S}_e$-approximation $M_t \to S_t$, there is a non-zero direct summand of $S_e[p]$ in $S_t$. We get a contradiction to Lemma \ref{Lemma3}, as $H^p(M_t) \to H^p(S_t)$ is non-zero.
\end{proof}

The following corollary is a direct consequence of propositions \ref{Prop0} and \ref{Prop1} and Theorem \ref{Theo1}.

\begin{Cor}
Assume that $A$ is finite dimensional positively graded and $S_e$ has a linear projective resolution. Then any two of the following conditions imply the third.
\begin{enumerate}[$(1)$]
    \item The global dimension of $A$ is finite.
\item The global dimension of $\Gamma$ is finite.
\item The algebra $Y(e)$ is finite dimensional and has finite global dimension. \end{enumerate}
\end{Cor}

Now, the next result is a particular case of the above corollary.

\begin{Cor}
Let $A$ be a finite dimensional Koszul algebra. Then any two of the following conditions imply the third.
\begin{enumerate}[$(1)$]
    \item The global dimension of $A$ is finite.
\item The global dimension of $\Gamma$ is finite.
\item The algebra $Y(e)$ is finite dimensional and has finite global dimension. \end{enumerate}
\end{Cor}

\section{Reduction for computing the Cartan determinant}

Assume that $A$ is a basic artin algebra. Thus, if $R$ denotes the center of $A$, then $R$ is an Artinian ring and $A$ is of finite length as an $R$-module. A long-standing conjecture in representation theory of artin algebras is the Cartan determinant conjecture. It states that if ${\rm gl.dim}\, A < \infty$, then the determinant of a Cartan matrix $C_A$ of $A$ is one. The reader is referred to \cite{Fuller} for a survey on this conjecture.

\medskip

Let $1 = e_1 + \cdots + e_n$ be a decomposition of $1$ into pairwise orthogonal primitive idempotents. Thus, if $P_i:=e_iA$ for $1 \le i \le n$, then $P_1, \ldots, P_n$ are the pairwise non-isomorphic indecomposable projective $A$-modules in $\mmod A$. Let $C_A$ be the Cartan matrix of $A$ associated to this decomposition with that order. The $(i, j)$ entry of $C_A$ is the length over $e_iAe_i$ of $e_jAe_i$. Alternatively, it is the length of $\Hom_A(P_i,P_j)$ over ${\rm End}_A(P_i)$. Of course, the matrix $C_A$ does depend on the chosen order $e_1, \ldots, e_n$ of $\{e_1, \ldots, e_n\}$, however, any Cartan matrix of $A$ is obtained from $C_A$ by simultaneous permutations of its rows and columns. Therefore, the determinant ${\rm cd}(A)$ of any Cartan matrix of $A$ is always the same and hence is a numerical invariant of $A$.  A well known result due to Eilenberg says that $C_A$ is invertible over $\Z$ whenever ${\rm gl.dim}\, A < \infty$; see \cite{Fuller}. Therefore, in this case, ${\rm det} C_A = {\rm cd}(A) = \pm 1$. So far, there is no known example of an artin algebra $A$ having finite global dimension with ${\rm cd}(A) = -1$.

\medskip

\begin{Prop}
Assume that ${\rm gl.dim}\, A < \infty$ and let $e$ be any idempotent of $A$ and put $\Gamma = (1-e)A(1-e)$. Assume further that ${\rm gl.dim}\, Y(e) < \infty$. Then ${\rm cd}(A) = {\rm cd}(\Gamma)$ where $\Gamma$ has finite global dimension. In particular, if the Cartan determinant conjecture holds for the smaller algebra $\Gamma$, then it holds for $A$.
\end{Prop}

\begin{proof}
We may assume that $e = e_1 + \cdots + e_r$ for some $1 \le r \le n$. Consider the Euler matrix $E := (C_A^{-1})^t$. It is well known that the $(i,j)$-entry of $E$ is
$$\sum_{t=0}^\infty(-1)^tc_i(\Ext_A^t(S_j, S_i))$$
where $S_t = e_tA/e_t{\rm rad}A$ for $1 \le t \le n$ and $c_i(\Ext_A^t(S_j, S_i))$ is the length of $\Ext_A^t(S_j, S_i)$ as an ${\rm End}_A(S_i)$-module, and it just counts the multiplicity of $e_iA$ in the $t$-term of a minimal projective resolution of $S_j$.
Let $X$ be the $r \times r$ submatrix of $C_A$ generated by the $r$ first rows and columns, and let $W$ be the $(n-r)\times(n-r)$ submatrix of $C_A^{-1}$ generated by the $n-r$ last rows and columns. A well-known formula in linear algebra, see \cite{GAE}, gives
$${\rm det} C_A = \frac{{\rm det} X}{{\rm det}W}$$
whenever ${\rm det}W$ is nonzero. Since $Y(e)$ has finite global dimension, a result of Wilson \cite{Wilson} states that the graded Cartan determinant ${\rm det}C_{Y(e)}(t)$ of $Y(e)$ is one. Putting $t  = -1$, we get ${\rm det}C_{Y(e)}(-1) = {\rm det}W^t=1$, so ${\rm det} W=1$. As, clearly, ${\rm det}X$ is the Cartan determinant ${\rm cd}(\Gamma)$ of $\Gamma$, we get
${\rm cd}(A) = {\rm cd}(\Gamma).$
\end{proof}

Note that we may take $e = 1$ in the above proposition. In this case, $\Gamma = 0$ and by convention, ${\rm cd}(\Gamma)=1$. Thus, if $Y(1)$ has finite global dimension, then the Cartan determinant conjecture is verified for $A$. Regarding the assumption of this proposition, it is generally easier to check the finiteness of the global dimension of $\Gamma$ rather than that of $Y(e)$. The following result holds provided Conjecture \ref{conj} holds.

\begin{Prop}[Assuming Conj.~\ref{conj}]
Assume that ${\rm gl.dim}\, A < \infty$ and let $e$ be an acyclic idempotent of $A$ and put $\Gamma = (1-e)A(1-e)$. Assume further that ${\rm gl.dim}\, \Gamma < \infty$. Then ${\rm cd}(A) = {\rm cd}(\Gamma)$. In particular, if the Cartan determinant conjecture holds for the smaller algebra $\Gamma$, then it holds for $A$.
\end{Prop}

\end{document}